\documentclass{article}
\usepackage[utf8]{inputenc}

\usepackage[square,sort,comma,numbers]{natbib}
\usepackage{graphicx}
\usepackage{amsmath}
\usepackage{amssymb}
\usepackage{amsthm}
\usepackage{tikz}
\usetikzlibrary{arrows}
\usetikzlibrary{positioning}
\usetikzlibrary{decorations.text}

\makeatletter
\DeclareFontEncoding{LS2}{}{\@noaccents}
\makeatother
\DeclareFontSubstitution{LS2}{stix}{m}{n}

\DeclareSymbolFont{largesymbolsstix}{LS2}{stixex}{m}{n}

\DeclareMathDelimiter{\lbrbrak}{\mathopen}{largesymbolsstix}{"EE}{largesymbolsstix}{"14}
\DeclareMathDelimiter{\rbrbrak}{\mathclose}{largesymbolsstix}{"EF}{largesymbolsstix}{"15}

\newtheorem{lemma}{Lemma}[section]
\newtheorem{prop}[lemma]{Proposition}
\newtheorem{corollary}[lemma]{Corollary}
\newtheorem{theorem}[lemma]{Theorem}

\newtheorem{thmm}{Theorem}

\newtheorem{thmk}{Theorem}

\theoremstyle{definition}

\newtheorem{Def}[lemma]{Definition}
\newtheorem*{notation}{Notation}

\theoremstyle{remark}

\newtheorem{example}[lemma]{Example}

\newenvironment{rules}
{\begin{list}{\(\cdot\)}{
\itemsep=0em
\leftmargin=5mm
\labelwidth=3mm
\labelsep=3mm
}}{\end{list}}

\def\RR{\mathbb{R}}
\def\ZZ{\mathbb{Z}}
\def\NN{\mathbb{N}}

\def\M{\mathcal{M}}

\def\root{\diamondsuit}

\let\temp\epsilon
\let\epsilon\varepsilon
\let\varepsilon\temp

\title{Closure Properties in the Class of Multiple Context-Free Groups.}
\author{Robert P. Kropholler and Davide Spriano}

\begin{document}
	
	\maketitle

\begin{abstract}
	We show that the class of groups with $k$-multiple context-free word problem is closed under graphs of groups with finite edge groups. 
\end{abstract}
\section{Introduction}

Multiple context-free languages (MCFLs) form a class of languages which contains context-free languages and is contained in context sensitive languages. MCFLs were introduced to better model natural languages for which it has been shown that context-free languages did not allow enough expressibility \cite{pollard_generalized_1984}. MCFLs allow some cross serial dependencies in natural languages such as Swiss German, for nice examples see \cite{salvati_notes}. They share several properties with context-free languages. Indeed, they form a cone of languages, they are semilinear, they are not closed under intersection and they satisfy some form of pumping lemma \cite{seki_multiple_1991}. MCFLs also have some useful decidability properties, for instance, one can decide membership in polynomial time \cite{seki_multiple_1991}. 

Given a presentation for a group \(G\), it is a natural question to ask whether two words represents the same element in \(G\). Using the elementary fact that \(g=h \leftrightarrow gh^{-1} = 1\), this is equivalent to establishing whether a given product of generators represents the identity element. 
One of the most successful strategies for tackling this question is to consider the set of all words that represents the trivial element, the so-called \emph{word problem}, and study it via language theoretical instruments. 
A remarkable result of Muller-Schupp \cite{muller_context-free_1981}, which relies on results of Stallings and Dunwoody \cite{stallings_group_1971, dunwoody_accessibility_1985} shows that the class of groups that have context-free word problem coincides with the class of virtually free groups.

With a complete classification of groups whose word problem is context-free, it is natural to look at larger classes. We will be interested in the class of multiple context-free languages (MCFLs), we will give a rigorous definition of this class in due course. The class was first studied in \cite{seki_multiple_1991}. The class of MCFLs is strictly larger than the class of CF languages. For example the language $\{a^nb^nc^n\mid n\in \NN\}$ is MCF but not CF. It was not until \cite{salvati_mix_2015} that it was known that the difference could be seen on the level of groups. Namely, \cite{salvati_mix_2015} shows that the word problem for $\ZZ^2$ is multiple context-free. However, since $\ZZ^2$ is not a virtually free group the word problem is not context-free. This result has been extended by Ho \cite{ho_word_2017_other} where it is shown that all free abelian groups have multiple context-free word problem. 

It is natural then to ask what are the closure properties of this class. It is shown in \cite{seki_multiple_1991}, that the class is closed under finite extensions and taking finitely generated subgroups. It is shown in \cite{gilman_groups_2017}, that the class is not closed under direct products.

In this paper we prove the following result:
\begin{thmk}\label{thm:graphofgroups}
	Let $G$ be the fundamental group of a finite graph of groups. Assume that all the vertex groups have multiple context-free word problem and all the edges groups are finite. Then $G$ has multiple context-free word problem. 
\end{thmk}

Since the class groups with regular word problem coincides with the class of finite groups, one could rephrase this result as saying that the class of MCF groups is closed under amalgamation over regular groups. This result is not true substituting regular groups with CF groups. Indeed, $F_2\times F_2 = (\ZZ^2\ast_{\ZZ}\ZZ^2)\ast_{F_2} (\ZZ^2\ast_{\ZZ}\ZZ^2)$ and does not have multiple context-free word problem. 

\section*{Acknowledgements} 
We greatly thank Bob Gilman for introducing us to the subject and making this project possible. This work was started at MSRI, Berkeley, where research is supported by the National Science Foundation under Grant No. DMS-1440140. The second author would like to thank UC Berkeley for inviting him as a visiting scholar. The first author would like to thank Alessandro Sisto for inviting me to complete this work at the ETH. We thank the anonymous referee for helpful comments and suggestions, in particular the addition of Section 6. Finally, we would like to thank Neil Fullarton for his invaluable work with a stapler.

\section{Background}

We are interested in the study of formal languages. In this section, we will give an introduction to formal languages and MCFLs. For a more comprehensive treatment, we refer to \cite{hopcroft_formal_1969}.
\begin{Def}
Given a finite set \(\Sigma\), \(\Sigma^*\) is the {\em free monoid over \(S\)}, i.e. the set of all finite words in \(\Sigma\) with the concatenation operation. We will denote with \(\epsilon\) the trivial element of \(\Sigma^*\), namely the empty word.
\end{Def}

\begin{Def}
Given a finite set \(\Sigma\), we say that a set \(L \subseteq \Sigma^*\) is a \emph{language} over \(\Sigma\). 
\end{Def}

Since the definition of language is very broad, we will restrict our attention to languages that have a nice description. The reader should think of this as the same meta-distinction between continuous functions \(\RR \rightarrow \RR\) and continuous functions that can be phrased in terms of elementary functions. 

Hence we want to prescribe a general recipe that will allow us to produce languages. 

\paragraph*{Chomsky grammars and hierarchy.}

\begin{Def}
A \emph{Chomsky grammar} \(G\) is a tuple \((\Sigma, N, \delta, S )\) where \(\Sigma\) and \(N\) are (disjoint) finite sets, \(S \in N\) and \(\delta\) is a finite subset of \(\big((\Sigma \cup N )^*\smallsetminus\Sigma^*\big) \times (\Sigma \cup N)^*\). Namely, if \((x,y) \in \delta\), then \(x\) contains at least one symbol of \(N\). 
We call \(\Sigma\) the set of  \emph{terminals} of \(G\), \(N\) the set of \emph{non terminals},  \(S\) the \emph{starting symbol} and \(\delta\) the \emph{production rules}.
\end{Def}

\begin{notation}
We will often use the following conventions: the elements of \(\Sigma\) will be denoted by lower case letters (ex. \(\{a,b,c\}\)), the elements of \(N\) by upper case letters (ex. \(\{A,B,S\}\)), and elements \(\tau = (aB, BccA)\) of \(\delta\) as \(\tau \colon aB \rightarrow BccA\). 
\end{notation}

Given a grammar \(G =(\Sigma, N, S, \delta)\), it is always possible to associate a (possibly empty) language \(L(G)\subset\Sigma^*\). We will describe inductively the language \(L(G)\).
\begin{Def}
Let \(G=(\Sigma, N, S, \delta)\) be a grammar. We want to describe a subset \(D(G) \subseteq (\Sigma \cup N)^*\) of \emph{derivable} words. 
\begin{rules}
\item \(S\) is derivable, 
\item for \(u,v,w \in (\Sigma \cup N)^*\), if \(uvw\) is derivable and the rule \(v \rightarrow x \) is an element of \(\delta\), then \(uxw\) is derivable. In particular, we say that \(uxw\) is \emph{derivable} from \(uvw\).
\end{rules}
We say that a \emph{derivation} (for \(w_k\)) is a chain of words \(S = w_1, \dots , w_k\) such that \(w_{i+1}\) is derivable from \(w_i\). 
The \emph{language associated} to the grammar \(G\) is the intersection \( L(G) = D(G) \cap \Sigma^*\), namely all the derivable words that consists only of terminals symbols. 
\end{Def}

\begin{example}
Let \(G = (\{ a,b,c\}, \{ A,B,S\}, S, \delta)\) be a grammar, where \(\delta\) consists of the following rules:
\begin{rules}
\item \(\tau_1 \colon S \rightarrow AB\),
\item \(\tau_2 \colon A \rightarrow aAb\),
\item \(\tau_3 \colon B \rightarrow ABc\),
\item \(\tau_4 \colon A \rightarrow \epsilon\),
\item \(\tau_5 \colon B \rightarrow \epsilon\).
\end{rules}
To generate the language \(L(G)\), we will try to understand the derivable words. We start with the symbol \(S\). The only rule we can apply at the first step is \(\tau_1\), yielding \(AB\). Then we can substitute \(A\) with \(aAb\), using rule \(\tau_2\), getting \(aAbB\). Applying \(\tau_2\)  \(k\) more times gives \(a^kAb^kB\). Rule \(\tau_4\) gives \(a^kb^kB\). 
Now, if we apply rule \(\tau_3\), we will get \(a^kb^kABc\). We can repeat the process above and get some word of the form \(a^{k_1}b^{k_1}\dots a^{k_n}b^{k_n}Bc^m\). After applying rule \(\tau_5\), we would get \(a^{k_1}b^{k_1} \dots a^{k_n}b^{k_n}c^m\), which is a string composed of non terminals only. 
\end{example}

We now give a classification of some grammars.

\begin{Def}
	A Chomsky grammar \(G = (\Sigma, N, \delta, S )\) is called:
	\begin{rules}
	\item \emph{regular} if all the elements of \(\delta\) have the form \(X \rightarrow wY\), where \(X \in N\), \(Y \in N \cup \{\epsilon\}\),
		 and \(w \in \Sigma^*\);
	\item \emph{context-free} if all the elements of \(\delta\) have the form \(X \rightarrow w\), where \(X \in N\) and \(w \in (\Sigma \cup N)^*\);
	\item \emph{unrestricted} otherwise.
	\end{rules}
	The language \(L (G)\) is \emph{regular} (respectively \emph{context-free} or \emph{recursively enumerable}) if \(G\) is regular (repectively context-free or unrestricted).  
\end{Def}

The intuitive idea that one should have about the above definition is the following: a derivation in a regular language consists of substituting the last letter of a word with a new string of letters. A derivation in a context-free language consists of substituting a single letter (but not necessarily the last one) of a word with a new string of letters. The last case covers all other possibilities.

The gap between being context-free and being recursively enumerable seems (and in fact is) very big. 
The class of multiple context-free languages (MCFLs) that we are going to describe, is one of the classes that properly lives in this gap, namely properly contains context-free languages, and is properly contained in the class of recursively enumerable languages \cite{seki_multiple_1991}.

As before, we are going to describe a grammar that defines the class of MCFLs. It should be noted that this will not be a Chomsky grammar. We  start with the definition of \emph{linear rewriting function}. The idea is very simple, but the definition may look a bit convoluted. Intuitively, a linear rewriting function is a function that ``paste words together", possibly adding some string of letters. For instance, if \(a,b\) are letters and \(v,w\) words, a linear rewriting function is \((v,w) \mapsto waabvb\). 

\begin{Def} 
	Fix a finite alphabet \(\Sigma\), and let \(X = \{x_1, \dots x_n\}\) be a finite (possibly empty) set of variables. 
	A \emph{rewriting} on the variables \(\{x_1, \dots , x_n\}\) is a word \(w\in (X \cup \Sigma)^*\). We say that a rewriting \(w\) is 	
	\emph{linear} if each element of \(X\) occurs at most once. 

	Given a rewriting \(w\), we can associate to it the function \(f_w \colon (\Sigma^*)^n \rightarrow \Sigma^*\) that  associates to each 
	tuple \((u_1, \dots , u_n)\) the word obtained substituting in \(w\) each occurrence of \(x_i\) with \(u_i\). If \(n=0\),
	then \((\Sigma^*)^0 = \{\epsilon\}\) and \(f_w\) is the constant function \(w\). 
	A rewriting function is \emph{linear} if it comes from a linear rewriting.

	We say that a function \(f \colon (\Sigma^*)^n \rightarrow (\Sigma^*)^m\) is a \emph{(multiple) rewriting function} if it is a rewriting 
	function in each component. A (multiple) rewriting function coming from rewritings \(w_1, \dots, w_m\) is \emph{linear} if
	\(w_1\dots w_m\) is linear.
\end{Def}

Note that being linear in each component is not enough for a multiple rewriting function to be linear. In fact, the whole word \(w_1 \dots w_m\) must be linear, this implies that each variable \(x_i\) appears in at most one of the \(w_j\). 
In order to simplify notation, from now on we will call multiple rewriting functions simply rewriting functions.

\begin{Def}
A \emph{stratified set} is a set \(N\) equipped with a function \(\Vert\cdot \Vert \colon N \rightarrow \NN\smallsetminus\{0\}\). The function $\Vert\cdot\Vert$ is called a {\em dimension}.
\end{Def}

\begin{Def}
	A \emph{multiple context-free grammar} (MCFG) on an alphabet \(\Sigma\) is a tuple \((\Sigma, N, S, F)\) satisfying the following:
	\begin{rules}
		\item \(\Sigma\) is a finite set of \emph{terminals}.
		\item \(N\) is a finite stratified set of \emph{non terminals}.
		\item \(S \in N\) is the \emph{starting symbol} such that \(\Vert S\Vert = 1\).
		\item \(F\) is a finite set of elements of the form \((A, f, B_1, \dots, B_s)\), where \(A, B_1, \dots , B_s\) 
			are elements of \(N\), and \(f\colon  (\Sigma^*)^{\Vert B_1\Vert + \dots +\Vert B_s\Vert} \rightarrow (\Sigma^*)^{\Vert A\Vert}\) is a linear 
			rewriting function.
	\end{rules}
	
	Given an element \(\tau = (A, f, B_1, \dots, B_s)\) of \(F\), we will denote it by
	 \(\tau= A \rightarrow f(B_1, \dots , B_s)\).
	
	We say that the grammar is $k$-MCF if $\Vert A\Vert\leq k$ for all $A\in N$. 
\end{Def}

As in the case of Chomsky grammars, given a MCFG \(H\), we want to associate a language \(L(H)\) to it. 
\begin{Def}
	Let \(H= (\Sigma, N, S, F)\) be a MCFG, and let \(A \in N\). 
	We inductively define \(D_H(A) \subseteq (\Sigma^*)^{\Vert A\Vert}\) as follows: for each \(\tau \in F\):
	\begin{rules}
		\item if \(\tau = A \rightarrow f (\epsilon)\), then \(f(\epsilon) \in D_H(A)\);
		\item if \(\tau = A \rightarrow f (B_1, \dots , B_s)\)	and \(y_1 \in D_H(B_1), \dots, y_s \in D_H(B_s)\), 
			then \(f(y_1, \dots, y_s) \in D_H (A)\). 
	\end{rules}
\end{Def}

\begin{Def} 
	For a MCFG \(H = (\Sigma, N, S, F)\), we define the \emph{language associated to \(H\)} as \(D_H(S)\). We say that a language \(L\) is a 
	\emph{multiple context-free language} if there is a MCFG \(H\) such that \(L = D_H(S)\).
\end{Def}

\section{Grammars and automata}

The goal of this section is to explain the relation between grammars and automata. 
In what follows, an automaton should be thought as a ``computer with limitations", namely as a machine that can do some operations, but does not possess  the power (usually memory) of a Turing machine. As in the case of grammars, an automaton is naturally associated to a language. The intuitive explanation for this is the following: an automaton is associated to an algorithm that, given a word, either ``accepts'' or ``rejects'' it. The language associated to an automaton is the set of all ``accepted'' words.

In what follows, we fix a finite alphabet \(\Sigma\), and all the definitions are understood to be dependent on \(\Sigma\).
Recall that a \emph{partial function} \(f\colon A \dashrightarrow B\) is a map of sets defined on a subset \(C\subseteq A\), called the \emph{domain} of \(f\).

\begin{Def}
A \emph{storage type} is a tuple \( T = (C,P,F,C_I)\) satisfying the following: \(C\) is a set, called the set of \emph{storage configurations}; \(P\) is a subset of the power set \(\mathcal{P}(C)\), and the elements of \(P\) are called \emph{predicates}; \(F\) is a set of partial functions \(f\colon C \dashrightarrow C\) called \emph{instructions}; and \(C_I \subseteq C\) is a set of \emph{initial configurations}.
\end{Def}

\begin{Def}
	An \emph{automaton with storage} is a tuple \(\M = (Q, T, I, \delta)\), where \(Q\) is a finite set of \emph{states}, \(T = (C,P,F,C_I)\) is a storage type, \(I\) is a tuple \(I= (q_I, c_I, Q_F)\) where \(q_I \in Q\) is the  \emph{initial state}, \(Q_F \subseteq Q\) are the \emph{final states}, and \(c_I \in  C_I\) is the \emph{initial storage configuration}.
	Finally \(\delta \subseteq Q \times (\Sigma \cup \{\epsilon\}) \times P \times F\times Q\) is a finite set of \emph{transitions}. 
\end{Def}

\begin{Def}
	Given an automaton with storage \(\M = (Q,T, I, \delta)\), we define the \emph{graph realisation of \(\M\)}, denoted by \(\Gamma (\M )\), as the 
	following oriented labelled graph:
	\begin{rules}
		\item The vertices of \(\Gamma (\M )\) are the elements of \(Q \times C\). 
		\item To each  $\tau = (q_1, \sigma , p , f , q_2) \in \delta$, we associate an oriented edge between the pair \(((q_1, c_1), (q_2, c_2))\) if 
		\(c_1 \in p\), \(f(c_1) = c_2\). In that case the label of this edge is \(\sigma\). 
	\end{rules}
\end{Def}

Note that \(f\) is a partial function, so with \(f(c_1)= c_2\) we are also asking that \(c_1\) is in the domain of \(f\).

\begin{Def}
	Let \(\Sigma\) be an alphabet, and let \(g \colon(\Sigma \cup \{\epsilon\})^* \rightarrow \Sigma^*\) be the morphism of monoids that 
	sends \(\epsilon\) to the empty word, and is the identity on all the other generators. 
	Given a word \(w\in \Sigma^*\) we say that a word \(w' \in (\Sigma \cup \{\epsilon\})^*\) is an \emph{\(\epsilon\)-expansion} of \(w\) if 
	\(g(w')=w\).  
\end{Def}

\begin{Def}
	Given an automaton with storage \(\M\) we define a language \(L (\M) \subseteq \Sigma^*\) as follows.
	A word \(w\) is in \(L (\M)\) if and only if there is an oriented path \(\gamma\) in \(\Gamma (\M)\) starting from \((q_I, c_I)\)
	and ending in a vertex \((q,c)\) with \(q \in Q_F\) such that the word formed by the labels of \(\gamma\) is an \(\epsilon\)-expansion
	of \(w\).
\end{Def}

In order to improve the readability of the above definitions, we will provide a fairy tale example to clarify the role of the various entities above.

Imagine there is a group of children playing a treasure hunt in a town. The town is finite (as towns tend to be) and each block of the town is one of the states \(Q\). The children possess an extremely bad memory, but luckily each of them is equipped with a book to write notes. The set \(C\) consists of all possible books with all possible contents opened to any page. {The set \(P\) contains some description about the state of the book, for example ``the set of all books open on a blank page" or ``all books open to the 12th page''.}

Now suppose that there is a voice guiding the game in order to help the children find the treasure, and in particular every now and then is reading out loud some hint (the alphabet \(\Sigma\)). The voice represents the word \(w\) in the alphabet. When a hint (letter) is read, the children will perform an action, and the possible actions are encoded in the set \(\delta\). 

At the start of the game, the children will all be in the central block of the city (\(q_I\)), with an empty book open on the first page (\(c_I\)), and the treasures will be buried in some blocks (\(Q_F\)) of the city. 
The typical turn will work as follows: every child will check on which block they are standing on (an element of \(Q\)), then listen to what the voice is saying (an element of \(\Sigma\)), and look if there is something written on the book (an element of \(P\)). Then each child decides which strategy apply on that turn (i.e. picks an element of \(\delta\)), which is compatible with the information \(Q, \Sigma\) and \(P\). Following such a strategy, they may change page or write something on the book (an element of \(F\)), and go to a new block (an element of \(Q\)) accordingly. If at any time a child cannot perform an action, then he or she is disqualified from the game. 
When the voice stops giving hints, each child will start digging exactly where they stand and see if a treasure is found.

If at least one child has found a treasure, then the instructions were correct (and hence the word \(w\) is accepted).

Let's start with some famous automata in order to familiarize with the above concepts.

\begin{Def}
	A \emph{trivial storage} is a storage type \(T = (C, P , F, C_I)\) with \( C = \{C_I\}\), $P = \{C\}$ and \(F= \{\texttt{id}\}\).
\end{Def}

\begin{Def}
	A \emph{finite state automaton} (FSA) is an automaton with storage with trivial storage. 
\end{Def}

It is a very easy exercise to see that a FSA is completely described by a finite oriented graph with edges labeled by elements of \(\Sigma\) (and not \(\Sigma \cup \{\epsilon\}\)). 

The following theorem forms a bridge between languages associated to grammars, and languages accepted by automata.

\begin{theorem}\cite{hopcroft_formal_1969}
	For a language \(L \subseteq \Sigma^*\) the following are equivalent:
	\begin{rules}
	\item \(L\) is associated to a regular grammar;
	\item \(L\) is accepted by a FSA.
	\end{rules}
\end{theorem}

\begin{Def}
	A \emph{push-down storage} over a finite alphabet alphabet \(\Omega\)  is a storage type \(T = (C, P, F, C_I)\) where: 
	\begin{itemize}
	\item \(C= \Omega^*\).
	\item We define the set \(\mathtt{equals}(\omega) \) as the set of words in \(\Omega^*\) that
	end with \(\omega\) (note that \(\mathtt{equals}(\epsilon)\) is the set \(\{\epsilon\}\)). 
	Then \(P= \{\mathtt{equals}(\omega) \mid \omega \in \Omega \cup \{\epsilon\}\}\).
	\item We define the function \(\mathtt{push}(\omega) \colon \Omega^* \rightarrow \Omega^*\) that sends \(x\) to \(x\omega\). 
		We also define a partial function \(\mathtt{pop}_\omega\colon \mathtt{equals}(\omega) \rightarrow \Omega^*\) that sends \(x\omega\) to \(x\).
		Then \(F = \{\mathrm{Id}\} \cup \{\mathtt{pop}_\omega , \mathtt{push}(\omega) \colon \omega \in \Omega\}\).
	\item \(C_I = \{\epsilon\}\).
\end{itemize}
\end{Def}

The intuitive idea behind the push-down storage is to have a stack of papers that can grow arbitrarily large, but the automaton can read only what is written on the top-most paper. This corresponds to the predicate \(\mathtt{equals}(\omega)\). Then one can put another paper on top with the letter \(\omega'\) (\(\mathtt{push}({\omega'})\)) or remove the old one (\(\mathtt{pull}_{\omega}\)). Note that the alphabet \(\Omega\) is, in general, not the same as \(\Sigma\).

\begin{Def}
	A \emph{push-down automaton} is an automaton with storage with push-down storage.
\end{Def}

\begin{theorem}\cite{chomsky_context-free_1962}
	For a language \(L \subset \Sigma^*\) the following are equivalent:
	\begin{rules}
	\item \(L\) is associated to a context-free grammar;
	\item \(L\) is accepted by a push-down automaton.
	\end{rules}
\end{theorem}

We now want to describe the last automaton we are interested in, namely the tree-stack automaton. 
\begin{Def}
	Let \(S\) be a set. If \(uv \in S^*\) we say that \(u\) is a \emph{prefix} for \(uv\). Given a set \(D \subseteq S^*\) we say that \(D\) is
	\emph{prefix-closed} if for each word \(w \in D\), all the prefixes of \(w\) are in \(D\).
	Similarly, we say that \(v\) is a \emph{suffix} for \(uv\).
\end{Def}
\begin{Def}
	Given an alphabet \(\Omega\), an \emph{\(\Omega\)-tree} is a partial function \(T\colon \NN^* \dashrightarrow \Omega \cup \{\root\}\) such that \(\mathrm{domain}(T) \subseteq \NN^*\) is prefix-closed and $T^{-1}(\root) = \{\epsilon\}$.
\end{Def}
Note that, this corresponds to a rooted tree, in the usual graph-theory sense, where each edge is labeled by a natural number, the root is labeled by the symbol \(\root\) and every other vertex is labeled by an element of \(\Omega\).

\begin{Def} 
	An \emph{\(\Omega\)-tree with a pointer} is a pair \((T, p)\) such that \(T\) is an \(\Omega\)-tree and \(p \in \mathrm{domain}(T)\).
\end{Def}
One should think of the pointer as a selected vertex of the tree. Figure 1 may provide some clarification.

\begin{figure}[h]
\begin{minipage}[c]{7cm}

\begin{tikzpicture}[
normal/.style={circle, draw=black, minimum size=5mm},
selected/.style={circle, draw=black, fill=gray!15, very thick, minimum size=5mm}]
\tikzstyle{level 1}=[sibling distance=35mm]
\tikzstyle{level 2}=[sibling distance=15mm]

\node(0)[normal]{\(\root\)}
	child{node[normal]{\(b\)}
		child{node[normal]{\(a\)} edge from parent node [left]{1}}
		edge from parent node [left]{1}}
	child{node[normal]{\(b\)}
		child{node[selected]{\(c\)} edge from parent node [left]{1}}
		child{node[normal]{\(a\)} edge from parent node [left]{2}}
		child{node[normal]{\(a\)} 
			child{node[normal]{\(c\)}edge from parent node [left]{1}}
		edge from parent node [left]{3}}
	edge from parent node [left]{2}};

\end{tikzpicture}
\end{minipage}
\begin{minipage}[c]{6cm}
\[T \colon \begin{cases}
\epsilon &\mapsto \root\\
1 &\mapsto b\\
11 &\mapsto a\\
2 &\mapsto b\\
21 &\mapsto c\\
22 &\mapsto a\\
23 &\mapsto a\\
231 &\mapsto c
\end{cases}\]
\end{minipage}
\caption{Graphic representation of\((T, 21)\)}
\end{figure}
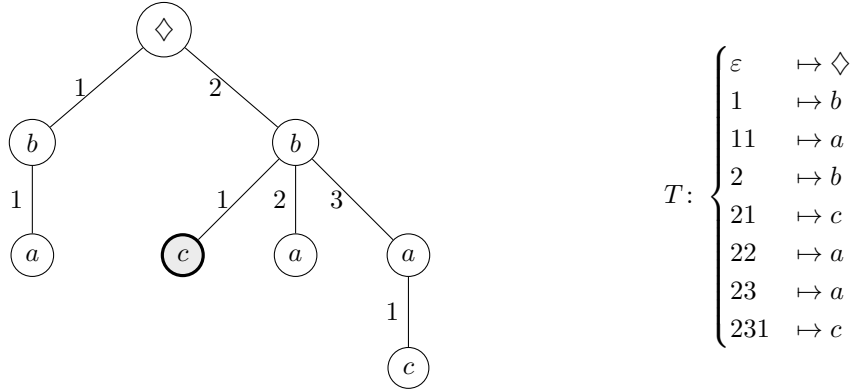

\begin{notation}
Let \(F \colon C \dashrightarrow X\) be a partial function, and let \(c \not\in\mathrm{domain}(F)\).
Then we define \(F[c\mapsto x]\) as the partial function defined on \(\mathrm{domain}(F) \cup \{c\}\), that agrees with \(F\) on \(\mathrm{domain}(F)\) and sends \(c\) to \(x\).
\end{notation}

\begin{Def}
	A \emph{tree-stack storage} over a finite alphabet alphabet \(\Omega\)  is a storage type \(T = (C, P, F, C_I)\) where: 
	\begin{itemize}
	\item \(C = \{ (T, p)\mid (T, p)\) is an \(\Omega\)-tree with pointer$\}$. 
	\item For \(\omega \in \Omega \cup \{\root\}\), we set \(\mathtt{equals}(\omega) = \{(T,p) \in C \mid T(p) = \omega\}\) and 
	\(\mathtt{notequals}(\omega) = \{(T,p) \in C \mid T(p) \neq \omega\}\).
	
	Then \(P= \{\mathtt{equals}(\omega), \mathtt{notequals}(\omega) \mid \omega \in \Omega \cup \{\root\}\}\cup \{C\}\).
	\item For \(n \in \NN\) and \(\gamma \in \Omega\), we define the following partial functions:
	\begin{rules}
	 	\item \(\texttt{push}_n({\gamma}) \colon \{(T,p) \mid pn \not\in \mathrm{domain}(T)\} \rightarrow C\) as the map \((T,p) \mapsto (T[pn\mapsto \gamma], pn)\).
	 	\item \(\texttt{up}_n \colon \{(T,p) \mid pn \in \mathrm{domain} (T)\} \rightarrow C\) as the map \((T,p)  \mapsto (T, pn)\).
	 	\item \(\texttt{down} \colon C-\texttt{equals}(\root) \rightarrow C\) as the map that sends \((T,pm) \mapsto (T,p)\), for \(m \in \NN\).
	 	\item \(\texttt{set}_{\gamma} \colon C - \texttt{equals}(\root) \rightarrow C\) as the map that sends \((T, p)\) to \((T',p)\), 
	 	where \(T'\) is obtained by \(T\)
	 	changing the value of \(p\) to \(\gamma\).
	\end{rules}
	Then \(F = \{\mathrm{Id}, \texttt{push}_n(\gamma), \texttt{up}_n, \texttt{down}, \texttt{set}_{\gamma} \mid \gamma \in \Omega, n \in \NN\}\). 
	\item \(C_I = \{(\epsilon \mapsto \root, \epsilon)\}\).
\end{itemize}
\end{Def}

One should not that the command $\texttt{push}_n(\gamma)$ can only be used if there is no branched labelled $n$ emanating form the vertex $p$. 

\begin{notation}
	For a subset \(F \) of \(\Omega\), we will write $\mathtt{notequals}(F)$ to indicate the finite union of \(\{\mathtt{noteequals} (\omega) \colon \omega \in F\}\).
	In particular, if we have the command \((q, a, \mathtt{notequals}(F), f, q')\) this will indicate the following finite set of rules \(\{ (q, a, \mathtt{notequals}(\omega), f, q') \colon \omega \in F\}\).
\end{notation}

\begin{Def}
A \emph{tree-stack automaton} is an automaton with storage with tree-stack storage.
\end{Def}

\begin{Def}\label{def: k-restricted}
	We say that a tree-stack automaton is {\em $k$-restricted} if for any $p\in \NN^*$, $n\in\NN$ and any path in the graph realisation $\Gamma(\mathcal{M})$ starting at $(q_I, c_I)$, the following holds.
	There are at most $k$ edges of the form $(q_1, (T_1, p))$ to $(q_2, (T_2, pn))$, where $q_1, q_2\in Q$ and $T_1, T_2$ are tree-stacks. 
\end{Def}
Intuitively, Definition \ref{def: k-restricted} states that every vertex in the tree-stack can be accessed from below a uniformly finite number of times. We will see in Lemma \ref{uniform} that this is equivalent to the fact that each vertex in the tree stack is only accessed for a uniformly bounded amount of time. 

\begin{theorem} \cite{denkinger_automata_2016}
	For a language \(L \subset \Sigma^*\) the following are equivalent:
	\begin{rules}
	\item \(L\) is associated to a $k$-MCFG;
	\item \(L\) is accepted by a $k$-restricted tree-stack automaton.
	\end{rules}
\end{theorem}

\begin{Def}
	A tree-stack automata is {\em cycle-free} if for every non-trivial loop in the graph realisation $\Gamma(\mathcal{M})$, there is at least one \texttt{push}, \texttt{up} or \texttt{down} command. 
\end{Def}

\begin{lemma}\cite{denkinger_automata_2016}
	Given a $k$-restricted tree-stack automaton $\mathcal{M}$ there exists a tree-stack $k$-restricted automaton $\mathcal{M}'$ such that $L(\mathcal{M}) = L(\mathcal{M}')$ and $\mathcal{M}'$ is cycle-free.
\end{lemma}

It is true that a \(1\)--restricted tree-stack automaton is equivalent to a push-down automaton. 
It is tempting to think that this equivalence can be realized just taking the stack of the push-down automaton as the tree-stack. 
However, this may often fail to be \(1\)--restricted. For instance, using a \texttt{pop} command followed by a \texttt{push} command would be seen as going down and up the tree. Repeating this we may visit a vertex arbitrarily often. Thus, this tree-stack automaton need not be $k$-restricted for any $k$. A $1$--restricted tree-stack automaton contains no \texttt{up} commands. Thus once a \texttt{down} command has been issued there is no way to return to the vertex that was left. 

\begin{example}
Let \(\mathcal{M}=(Q,T,I, \delta)\) be a push-down automaton over a finite alphabet \(\Omega\). We want to define a \(1\)--restricted tree-stack automaton \(\mathcal{N}\) such that \(L(\mathcal{M}) = L(\mathcal{N})\).

We define \(\mathcal{N}=(Q', T', I', \delta')\)  to be the following tree-stack automaton.
\begin{rules}
\item For each element \(\omega \in \Omega\), let \(\square_\omega\) be an extra symbol. Then \(Q'= Q \cup \{(\square_\omega,q)\mid \omega \in \Omega, q \in Q\}\).
\item \(T'\) is the tree-stack storage with respect to an \(\Omega'\)--tree, where \(\Omega' = \Omega \cup \{\star\}\).
\item The initial and final states of \(I'\) are the same as \(I\) (because \(Q\subseteq Q'\)).
\item \(\delta'\) will be the set containing the following instructions:
\begin{enumerate}
    \item For each rule \(\tau = (q_1, \sigma, p, f, q_2) \in \delta\) there is a corresponding rule \(\tau'= (q_1, \sigma, p', f', q_2') \in \delta'\) as follows. If \(p = \mathtt{equals}(\omega)\), then \(p'= \mathtt{equals}(\omega)\) (note that those predicates have the same names, but are subsets of different power sets). Similarly, if \(p\) represents the whole set of configurations of the push-down storage, then \(p'\) will represent the whole set of configuration of the tree-stack storage. 
    If \(f=\mathtt{push}(\omega)\), then \(f'=\mathtt{push}_0(\star)\) and  \(q_2'= (\square_\omega,q_2)\).
    If \(f=\mathtt{pop}_\omega\), then \(f'=\mathtt{down}\) and \(q_2'= q_2\).
    \item For the state \((\square_\omega, q)\) we have the instruction \(((\square_\omega, q), \epsilon, C, \mathtt{push}_1(\omega), q)\).
    \item For every \(q \in Q\), we have the instruction: \((q,\epsilon, \mathtt{equals}(\star), \mathtt{down}, q)\).
\end{enumerate}
\end{rules}
\end{example}

We also include an application of this example to give a tree-stack automaton which recognises the word problem in $\ZZ$. 

\begin{example}
	Define a tree stack automaton as follows
	
	$Q = \{S, q_f, q_t, q_T\}$, 
	
	$\Sigma$ is the alphabet $\{t, T\}$. 
	
	$T$ is a tree stack with alphabet $\{t, T, \square\}$. 
	
	$S$ is the start state with empty stack as the start stack. 
	
	$Q_f = \{q_f\}$. 
	
	$\delta$ consists of the commands
	\begin{align*}
	&(S, t, \texttt{notequal}{(T)}, \texttt{push}_0(\square), q_t)\\
	&(S, \epsilon, C, \texttt{push}_1(t), S)\\
	&(S, T, \texttt{notequal}{(t)}, \texttt{push}_0(\square), q_T)\\
	&(S, \epsilon, C, \texttt{push}_1(T), S)\\
	&(S, \epsilon, \texttt{equal}{(\square)}, \texttt{down}, S)\\
	&(S, t, \texttt{equals}{(T)}, \texttt{down}, S)\\
	&(S, T, \texttt{equals}{(t)}, \texttt{down}, S)\\
	&(S, \epsilon, \texttt{notequal}{(\root)}, \mathrm{Id}, q_f)
	\end{align*}
	
	This automaton accepts word which contain an equal number of the letter $t$ and the letter $T$. This is the word problem in $\ZZ$. 
\end{example}

The key point in both these examples is the fact that there is no $\texttt{pop}$ command for tree-stack automaton. We mimic the \texttt{pop} command by pushing $\square$ once a $\square$ is returned to a branch of the tree is no longer accessible, essentially popping everything on that branch.

For some explicit examples of 2-restricted tree stack automata see Examples 3.2 and 3.3 in \cite{denkinger_automata_2016}.



\section{Closure under free products}

In this section we prove that the class of groups whose word problem is multiple context-free is closed under free products. To do this we will show that given $G_1$ and $G_2$ with multiple context-free word problem we can construct a tree-stack automaton which accepts the word problem for $G_1\ast G_2$.

\begin{lemma}\label{lem:emptytree}
	Let $\mathcal{M}$ be a tree-stack automaton accepting the language $M$. Then there exists a tree-stack automaton $\mathcal{M}'$ such that $L(\mathcal{M}') = L$ and $\mathcal{M}'$ accepts a non-empty word only if the tree-stack storage is in the state $(T, \epsilon)$ for some $\Omega$-tree $T$. 
\end{lemma}
\begin{proof}
	We build a new automaton which accepts the same language as follows. 
	
	Add two extra states $q_f, \bar{q}_f$ to our automaton. We add the following transitions to $\delta$.  
	
	\begin{align*}
	&(q, \epsilon, C, \mathrm{Id}, q_f),  \forall q\in Q_F\\
	&(q_f, \epsilon, C, \texttt{down}, q_f)\\
	&(q_f, \epsilon, \texttt{equals}({\root}), \bar{q}_f)
	\end{align*}
	
	We change the set of accept states to $\{\bar{q}_f\}$.
	The language accepted by this new automaton is the same language as before. It should be noted that the new automaton has a single accept state and if $\mathcal{M}$ was cycle-free, then so is $\mathcal{M}'$. 
\end{proof}

It will also be useful to know that the amount of time spent at any vertex in the tree-stack is uniformly bounded.

\begin{Def}
	A {\em run} in a tree-stack automaton is a path in the graph realisation. This can be seen as a valid sequence of instructions. 
	
	An {\em accepted run} is a run which ends in an accept state.
\end{Def}

\begin{lemma}\label{uniform}
	If $M$ is a $k$-restricted cycle-free tree-stack automaton, then there is an $n$ such that, for each $p\in\NN^*$ and each path in the graph realisation of $\mathcal{M}$ starting at $(q_I, c_I)$, there are at most $n$ vertices in the run of the form $(q, (T, p))$, where $q$ and $T$ may vary. 
\end{lemma}
\begin{proof}
	Consider the two possibilities for entering a vertex of the form $(q, (T, p))$, where $p$ is fixed and $q$ and $T$ may vary. Either we have an edge $(q_1, (T_2, pm))\to (q, (T, p))$ or $(q_2, (T_2, \bar{p}))\to (q, (T, p))$, where $\bar{p}l = p$ for some $l$. There are only $k$ possibilities of the second instance since the automaton is $k$-restricted. 
	
	In the first instance, there must have been an edge of the form $(q', (T', p))\to (q'', (T'', pm))$ previosuly in the path. There are at most $k$ such edges by $k$-restrictedness. Since $\delta$ is finite there can only be a finite number of instructions that contain a \texttt{push} command. Therefore, there are a bounded number of choices for $m$. 
	
	We will not require the exact bound, however, it can be calculated. A good estimate is $k\times$(number of \texttt{push} commands)$\times$(length of the longest path in the automaton with no movement in the tree).  
\end{proof}

Let $G_1, G_2$ be groups with multiple context-free word problem, we now create the automaton which will accept the word problem for $G_1\ast G_2$. 
Ideally, one would like to take the ``free product'' of the automata. However, this will result in something infinite. The key idea is to do this at the level of the tree-stack storage only.

\begin{theorem}\label{thm:free_product}
	If $G_1$ and $G_2$ are groups with multiple context-free word problem, then $G_1\ast G_2$ has multiple context-free word problem.
\end{theorem}
\begin{proof}
	Let $W_i$ be the word problem in $G_i$ and $W$ be the word problem in $G_1\ast G_2$. Let $\mathcal{M}_i = (Q_i, T_i, I_i, \delta_i)$, where $T_i$ is a tree-stack storage over the alphabet $\Omega_i$ and $I_i = (q_I^i, c_I^i, Q_F^i = \{q_f^i\})$ be an automaton recognising the language $W_i$. 
	
	We will assume that these automata are $k$-restricted, cycle-free and accept a word if and only if the stack pointer is at the root. Let $n$ be the maximum of the two bounds obtained from Lemma \ref{uniform} applied to $\mathcal{M}_1$ and $\mathcal{M}_2$.

	We now define the automaton $\mathcal{M}$ that will recognize the language $W$.
	
	The states of $\mathcal{M}$ are $Q = Q_1\sqcup Q_2\sqcup\{S, F\}$, the storage type $T$ is the set of tree-stacks on the alphabet $\Omega = \Omega_1\sqcup\Omega_2\sqcup(Q\times\{\square_1, \square_2\})$. 
	The initial state is $S$, with empty initial tree and the final state is $F$. The transitions are $\delta = \delta_1'\sqcup\delta_2'\sqcup \delta_3$, where each set will be described shortly. Intuitively, the set \(\delta_3\) regulates the transitions between the two original automata, and we will obtain \(\delta_i'\) from \(\delta_i\) by substituting each instruction in \(\delta_i\) that contains a \(\root\) symbol with a finite set of instructions, one for each state of \(Q_i\). More precisely \(\delta_i' = (\delta_i\smallsetminus (\mathcal{D}_=\cup\mathcal{D}_{\neq}))\cup \mathcal{S}_=\cup\mathcal{S}_{\neq}\) where,
	\begin{rules}
		\item $\mathcal{D}_= =  \{(q_1, \sigma, \texttt{equals}(\root), f, q_2)\}$
		\item $\mathcal{D}_{\neq} = \{(q_1, \sigma, \texttt{notequals}(\root), f, q_2)\}$
		\item $\mathcal{S}_= =  \{(q_1, \sigma, \texttt{equals}((q, \square_i)), f, q_2)\mid (q_1, \sigma, \texttt{equals}(\root), f, q_2)\in \mathcal{D}_=, q\in Q \}$
		\item $\mathcal{S}_{\neq}  = \{(q_1, \sigma, \texttt{notequals}((q, \square_i)), f, q_2)\mid (q_1, \sigma, \texttt{notequals}(\root), f, q_2)\in \mathcal{D}_{\neq}, q\in Q \},$
	\end{rules} { and }
	\begin{align*}
	\delta_3 = \{&(S, \epsilon, \texttt{equals}(\root), \mathrm{Id}, F), \\
				&(S, \epsilon, C, \texttt{push}_1((S, \square_1)), q_I^1), 
				(S, \epsilon, C, \texttt{push}_2((S, \square_2)), q_I^2), \\
				&(q_f^1, \epsilon, \texttt{equals}((S, \square_1)), \texttt{down}, S), 
				(q_f^2, \epsilon, \texttt{equals}((S, \square_2)), \texttt{down}, S)\}
				\cup\\
				\{&(q, \epsilon, \mathtt{notequals}(Q_1\times\{\square_1\}), \mathtt{push}_i((q, \square_2)), q_I^2)\mid
				, q'\in Q_2, i\in \{-1, \dots, -n\}\}\cup\\
				\{&(q, \epsilon, \mathtt{notequals}(Q_2\times\{\square_2\}), \mathtt{push}_i((q, \square_1)), q_I^1)\mid q'\in Q_1,  i\in \{-1, \dots, -n\}\}\cup\\
				\{&(q_f^1, \epsilon, \mathtt{equals}((q', \square_1)), \mathtt{down}, q')\mid  q'\in Q_2\}\cup\\
				\{&(q_f^2, \epsilon, \mathtt{equals}((q', \square_2)), \mathtt{down}, q')\mid  q'\in Q_1\}.
	\end{align*}

	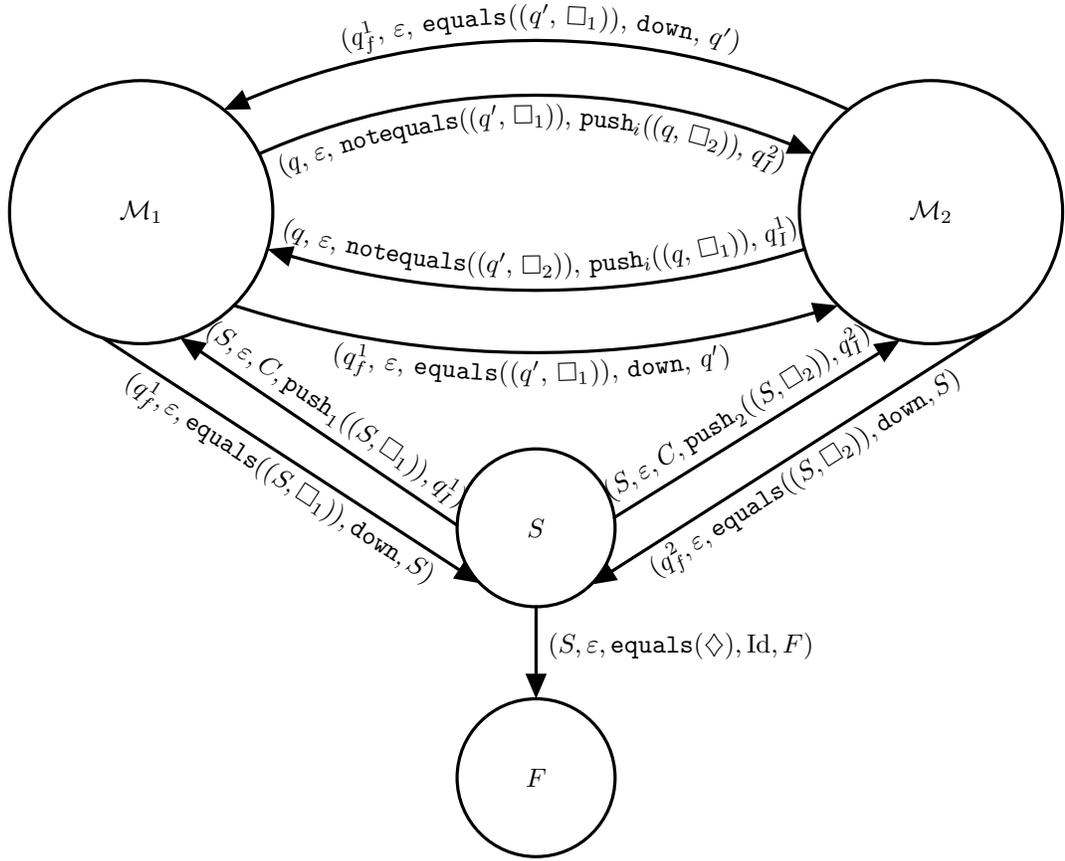
\begin{figure}
		\center
		\begin{tikzpicture}
		[line cap=round,line join=round,>=triangle 45,x=3.5cm,y=3.5cm]
		\draw [line width=1.2pt] (-1.5,1) circle (0.5);
		\draw [line width=1.2pt] (1.5,1) circle (0.5);
		\draw [line width=1.2pt] (0,-0.2) circle (0.3);
		\draw [line width=1.2pt] (0,-1.15) circle (0.3);
		\draw [->,line width=1.2pt] (0,-0.5) -- (0,-0.85);

		\def\myshift#1{\raisebox{-2.5ex}}
		\draw [->,line width=1.2pt, 
		postaction={decorate, decoration={text along path, text align=center, text={|\myshift|{$(q_f^1, \epsilon, \texttt{equals}((S, \square_1)), \texttt{down}, S)$}}}}]  (-1.6461074688950164,0.5218236648127683) -- (-0.21344987074194324,-0.4108059597835121);

		\def\myshift#1{\raisebox{1.5ex}}
		\draw [<-,line width=1.2pt, 
		postaction={decorate, decoration={text along path, text align=center, text={|\myshift|
					{$(S, \epsilon, C, \texttt{push}_1((S, \square_1)), q_I^1)$}}}}
		] (-1.3518098702909491,0.5224649903338869) -- (-0.2998502187655847,-0.19052127085359105);

		\def\myshift#1{\raisebox{-2.5ex}}
		\draw [<-,line width=1.2pt, 
		postaction={decorate, decoration={text along path, text align=center, text={|\myshift|{$(q_f^2, \epsilon, \texttt{equals}((S, \square_2)), \texttt{down}, S)$}}}}
		] (0.2084256897656984,-0.4157747247610175) -- (1.7423198889926752,0.5626430846567332);

		\def\myshift#1{\raisebox{1.5ex}}
		\draw [->,line width=1.2pt, 
		postaction={decorate, decoration={text along path, text align=center, text={|\myshift|
					{$(S, \epsilon, C, \texttt{push}_2((S, \square_2)), q_I^2)$\quad}}}}
		]  (0.29739645234462686,-0.16056207240700826) -- (1.3824773286814853,0.514007796640563);

		\def\myshift#1{\raisebox{-2.5ex}}
		\draw [->, shift={(0,4.227813824947694)},line width=1.2pt, 
		postaction={decorate,decoration={text along path,text align=center,text={|\myshift|
					{$($}{$q_f^1$}{, }{$\epsilon$}{, }{\texttt{e}}{\texttt{q}}{\texttt{u}}{\texttt{a}}{\texttt{l}}{\texttt{s}}{$($}{$($}{$q'$}{, }{$\square_1$}{))}{, }{\texttt{d}}{\texttt{o}}{\texttt{w}}{\texttt{n}}{, }{$q'$}{) }{ }
		} } } ]  plot[domain=4.402340010083532:5.0224379506858465,variable=\t]({1*3.762152015453168*cos(\t r)+0*3.762152015453168*sin(\t r)},{0*3.762152015453168*cos(\t r)+1*3.762152015453168*sin(\t r)});

		\def\myshift#1{\raisebox{1.5ex}}
		\draw [<-, shift={(0,4.106118077870352)},line width=1.2pt, 
		postaction={decorate,decoration={text along path,text align=center,text={|\myshift|
					{$($}{$q$}{, }{$\epsilon$}{, }{\texttt{n}}{\texttt{o}}{\texttt{t}}{\texttt{e}}{\texttt{q}}{\texttt{u}}{\texttt{a}}{\texttt{l}}{\texttt{s}}{$($}{$($}{$q'$}{, }{$\square_2$}{))}{, }{\texttt{p}}{\texttt{u}}{\texttt{s}}{\texttt{h}}{$_i$}{(}{(}{$q$}{, }{$\square_1$}{)}{)}{, }{$q_I^1$}{)}
		} } } ]  plot[domain=4.4079423576393015:5.016835603130078,variable=\t]({1*3.403602230598923*cos(\t r)+0*3.403602230598923*sin(\t r)},{0*3.403602230598923*cos(\t r)+1*3.403602230598923*sin(\t r)});

		\def\myshift#1{\raisebox{-2.5ex}}
		\draw [->, shift={(0,-1.15)},line width=1.2pt,
		postaction={decorate,decoration={text along path,text align=center,text={|\myshift|
					{$($}{$q$}{, }{$\epsilon$}{, }{\texttt{n}}{\texttt{o}}{\texttt{t}}{\texttt{e}}{\texttt{q}}{\texttt{u}}{\texttt{a}}{\texttt{l}}{\texttt{s}}{$($}{$($}{$q'$}{, }{$\square_1$}{))}{, }{\texttt{p}}{\texttt{u}}{\texttt{s}}{\texttt{h}}{$_i$}{(}{(}{$q$}{, }{$\square_2$}{)}{)}{, }{$q_I^2$}{) }{ }
		} } } ]  plot[domain=1.9882313635053133:1.15336129008448,variable=\t]({1*2.594693225303058*cos(\t r)+0*2.594693225303058*sin(\t r)},{0*2.594693225303058*cos(\t r)+1*2.594693225303058*sin(\t r)});

		\def\myshift#1{\raisebox{1.5ex}}
		\draw [<-, shift={(0,-1.15)},line width=1.2pt, postaction={decorate,decoration={text along path,text align=center,text={|\myshift|
					{$($}{$q_f^1$}{, }{$\epsilon$}{, }{\texttt{e}}{\texttt{q}}{\texttt{u}}{\texttt{a}}{\texttt{l}}{\texttt{s}}{$($}{$($}{$q'$}{, }{$\square_1$}{))}{, }{\texttt{d}}{\texttt{o}}{\texttt{w}}{\texttt{n}}{, }{$q'$}{)}
		} } } ]  
		
		plot[domain=2.0078378898065155:1.1337547637832779,variable=\t]
		({1*2.802804589490702*cos(\t r)+0*2.802804589490702*sin(\t r)},{0*2.802804589490702*cos(\t r)+1*2.802804589490702*sin(\t r)});

		\draw[color=black] (-0,-1.15) node {$F$};
		\draw (0,-0.2) node {$S$};
		\draw (-1.5,1) node {$\mathcal{M}_1$};
		\draw (1.5,1) node {$\mathcal{M}_2$};
		\draw (0.55,-0.65) node {$(S, \epsilon, \texttt{equals}(\root), \mathrm{Id}, F)$};
		
		\end{tikzpicture}
		\caption{A depiction of the automata accepting the word problem of $G_1\ast G_2$}
		\label{automaton}
	\end{figure}
	
	\begin{figure}
		\center
		\setlength{\unitlength}{300pt}
		\begin{picture}(1,0.53298141)%
		\put(0,0){\includegraphics[width=\unitlength]{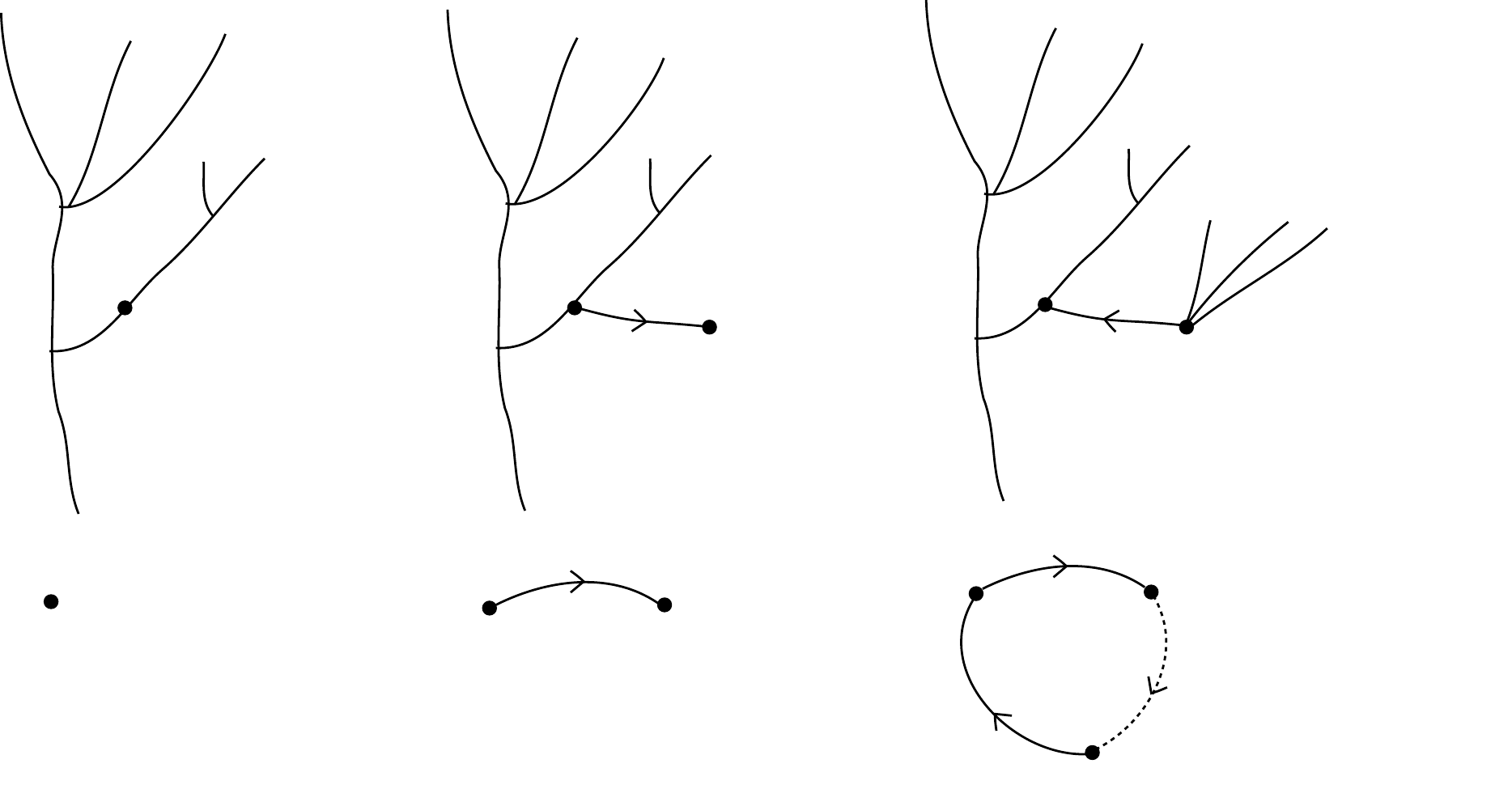}}%
		\put(0.05290225,0.11580525){\color[rgb]{0,0,0}\makebox(0,0)[lb]{\smash{$q_1$}}}%
		\put(0.32377827,0.1019959){\color[rgb]{0,0,0}\makebox(0,0)[lb]{\smash{$q_1$}}}%
		\put(0.64351822,0.10624489){\color[rgb]{0,0,0}\makebox(0,0)[lb]{\smash{$q_1$}}}%
		\put(0.43425319,0.10093362){\color[rgb]{0,0,0}\makebox(0,0)[lb]{\smash{$q_I^2$}}}%
		\put(0.76558192,0.1561631){\color[rgb]{0,0,0}\makebox(0,0)[lb]{\smash{$q_I^2$}}}%
		\put(0.0946912,0.30552472){\color[rgb]{0,0,0}\makebox(0,0)[lb]{\smash{$\gamma$}}}%
		\put(0.69216969,0.29362742){\color[rgb]{0,0,0}\makebox(0,0)[lb]{\smash{$\gamma$}}}%
		\put(0.37391691,0.29511458){\color[rgb]{0,0,0}\makebox(0,0)[lb]{\smash{$\gamma$}}}%
		\put(0.4735568,0.28321728){\color[rgb]{0,0,0}\makebox(0,0)[lb]{\smash{$(q_1, \square_2)$}}}%
		\put(0.79329675,0.28173012){\color[rgb]{0,0,0}\makebox(0,0)[lb]{\smash{$(q_1, \square_2)$}}}%
		\put(0.20649636,0.36505125){\color[rgb]{0,0,0}\makebox(0,0)[lb]{\smash{$\to$}}}%
		\put(0.5244283,0.37520624){\color[rgb]{0,0,0}\makebox(0,0)[lb]{\smash{$\to$}}}%
		\put(0.72982811,0.00526098){\color[rgb]{0,0,0}\makebox(0,0)[lb]{\smash{$q_f^2$}}}%
		\end{picture}%
		\caption{The process of opening a new tree and returning to an old tree once the word is accepted.}
		\label{treestackfig}
		
	\end{figure}
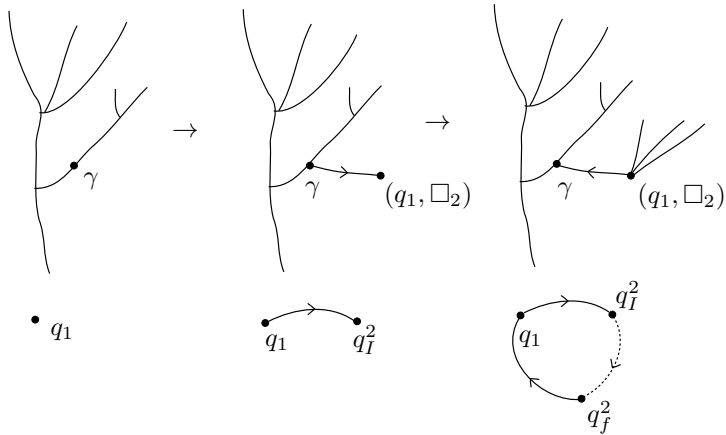
	
	The reader should note that tree-stacks were defined with $\NN$ and negative labels have been used above. One should note that $\ZZ$ is countable so the labels can be made positive. 

	The automaton above is $k$-restricted since the commands in $\delta_3$ do not add any $\texttt{up}_n$ commands and all such commands come from the automata $\mathcal{M}_i$ which are $k$-restricted.

	We want to show that $W = L(\mathcal{M})$.
	
	The way the automaton above works is as follows. We start with our word and move to one of the automata $\mathcal{M}_1$ or $\mathcal{M}_2$, say $\mathcal{M}_1$. We then read a word in $\Sigma_1$ and move in this automaton as usual. When we come to a letter from $\Sigma_2$ we move to the automaton $\mathcal{M}_2$ recording the state $q\in Q_1$ where we left $\mathcal{M}_1$ and opening a new branch on the tree. Later we will read a letter of $\Sigma_1$, if we do this at the final state of $\mathcal{M}_2$ then we move back to $q$, otherwise we open a new branch and move to $q_I^1$ and continue this process. 
	
	An accepted run $\Lambda$ of the automaton will have the pointer start and end at the root of the tree-stack. Let $T_f$ be the final tree-stack for the run. We can colour the non-root vertices of $T_f$ red and blue as follows. Colour a vertex red if the label is from $\Omega_1\cup (Q_2\times\{\square_1\})$ and blue otherwise. Note that after each instruction there is a tree-stack which embeds, as a graph, into $T_f$. Since the only \texttt{set} commands are to be found in $\delta_1'$ and $\delta_2'$, one could colour a vertex upon creation, the above embedding will then be colour preserving. 
	
	There is a subtree $T_c\subset T_f$ of a single colour whose complement is connected.

	For each instruction there are two possible pointers, these can be viewed as vertices of $T_f$. Let $\Theta$ be the instructions in $\Lambda$ such that both pointers are in $T_c$. We claim that all the elements of $\Theta$ are consecutive. This is because there are no $\texttt{up}$ commands with negative labels, so once we leave $T_c$ there is no way to return. Note that \(\Theta\) start at the initial state of one of the automaton and ends at the corresponding final state. In particular, it can be viewed as an accepted run in \(\mathcal{M}_i\) and the subword \(v\) of the run \(\Lambda\) associated to \(\Theta\) is an element of $W_i$. 
	
	Using the above,  $\Lambda$ decomposes as $\Lambda_1\theta_1\Theta\theta_2\Lambda_2$, where $\theta_1\in \delta_3$ is an instruction containing a $\texttt{push}$ command and $\theta_2\in\delta_3$ is an instruction containing a $\texttt{down}$ command. 
	However, to leave the tree \(T_c\), \(\theta_1\) and \(\theta_2\) pair up, by which we mean that the state of the automaton and the pointer before \(\theta_1\) and after \(\theta_2\) are the same. Also, the tree-stacks outside \(T_c\) remains unchanged.
	Thus, \(\Lambda_1 \Lambda_2\) is an accepted run of $\mathcal{M}$. As a consequence, we have that the word \(w\) corresponding to the run \(\Lambda\) decomposes as \(w_1 v w_2\), where  \(v\) is an element of \(W_i\) and \(w_1w_2\) is accepted by \(\mathcal{M}\). 
	By considering words that are trivial in $G_1\ast G_2$, we have that if \(w_1w_2\) is an element of $W$, then so is \(w\). 

	For the base case, note that if \(T_c=T_f\), then $w\in W_i$. 
	Thus by induction on the number of maximal one-colored subtrees, \(L(\mathcal{M})\) is a subset of $W$.

	For the other direction, we will use induction on the free product length of the word $w\in W$. The {\em free product length} of $w$ is the $p$ such that $w = w_1w_2\dots w_p$ and if $w_i\in W_j$, then $w_{i+1}\notin W_j$. 

	It is clear that words of free product length 1 are in the language $L(\mathcal{M})$.
	
	If $w = w_1\dots w_p$ has free product length $p$ and is an element of $W$, then there is an $i$ such that $w_i$ is an element of $W_j$. We will assume that $w_i\in W_1$. The run the machine will take is as follows, make the run for the word $w_1\dots w_{i-1}w_{i+1}\dots w_p$, which exists by induction hypothesis. At the point where the word $w_i$ is read we will open a new tree and move to the automaton $\mathcal{M}_1$ following a run for this word. 
	
	This run will finish at the root of the new tree and then return to the automaton $\mathcal{M}_2$ to continue the run where it left off. 
	
	To make sure that we can do this process we have to be able to $\mathtt{push}$ a new edge at the correct moment. This may not be possible if we have already pushed $n$ edges at this vertex. However, we assumed that the automaton $\mathcal{M}_1$ can only spend a uniformly bounded amount of time at any vertex and we added more $\mathtt{push}$ commands than this bound. Thus, there will always be a run for the word $w_1\dots w_{i-1}w_{i+1}\dots w_p$, where we can make a $\mathtt{push}$ at the desired moment. 
\end{proof}

In fact in the proof we have shown a slightly stronger result. 

\begin{corollary}
	If $G_1$ and $G_2$ are groups whose word problem is $k$-MCF, then the word problem in $G_1\ast G_2$ is $k$-MCF.
\end{corollary}

\begin{proof}
	It is clear from the proof of Theorem \ref{thm:free_product} that the automaton constructed is \(\mathrm{max}\{k_1, k_2\}\)-restricted. 
	Indeed, all the instruction that contains commands \(\texttt{up}\) are contained in \(\delta'_1 \cup \delta'_2\). Applying instructions contained in \(\delta'_i\) will not move the pointer to a vertex of a different colour (where the colouring is defined as in the proof of Theorem \ref{thm:free_product}). Thus, if a vertex is contained in the interior of a one-colored subtree, say the colour corresponding to \(W_1\), then that vertex will satisfy the \(k_1\)-restriction condition. 
\end{proof}

\section{Amalgamated Free Products}

In this section we generalize the previous result to show that the class of groups with multiple context-free word problem is closed under amalgamation over finite subgroups. 

The idea is similar to the previous proof, there are however more details. We feel that the interested reader should understand the proof of Theorem \ref{thm:free_product} which encapsulates most of the details in an easier setting. The key idea is the following:

\begin{prop}\label{prop:finitesubgroup}
	Let $G$ be a group with multiple context-free word problem. Let $H$ be a finite subset of $G$. Then $\{w\in\Sigma^*\mid w$ represents an element of $H\}$ is a multiple context-free language.
\end{prop}
\begin{proof}
	For each \(h \in H\), let $v_h$ be a word representing $h^{-1}$ in $\Sigma$. Let $R = \{v_h\mid h\in H\}$. Since \(H\) is a finite set, so is \(R\). Let $R'$ be the set of (possibly empty) suffixes of words in $R$. Let $\mathcal{M} = (Q, T, I, \delta)$ be an automaton recognising the word problem in $G$ with start state $q_I$ and a single final state $q_f$, where $T$ is the set of tree-stacks over the alphabet $\Omega$. Assume that this automaton has been modified as in Lemma \ref{lem:emptytree}.
	
	The idea is the following: let \(w\) be the input word.
	We will build an automaton that will ``guess'' an element of \(H\), say \(h\), and then proceed to process the word \(v_h w\) in \(\M\).
	The way this is done, is by adding a "second variable" to the states. The second variable represents the new word that is inserted.
	If the second variable is empty, then the automaton acts exactly as before. 
	Otherwise, if the automaton is in a state \((q, v)\), where \(v = a_1 \dots a_n\) is a (non trivial) word,
	the automaton acts as if it was in the state \(q\) and the first letter of \(v\) (that is, \(a_1\)) is read. Then the second variable
	becomes \(a_2 \dots a_n\). 
	
	More formally.
	We will build a new automaton $\mathcal{M}' = (Q', T', I', \delta')$ as follows. The set of states $Q'$ will be $(Q\times{R'})\sqcup\{S\}$. The storage $T'$ will be tree-stacks over $\Omega$. The set of transitions $\delta'$ will consist of four types of transformation:
	\begin{flalign*}
		&&(S, \epsilon, \texttt{equals}(\root), \text{Id}, (q_I, v)) &&&\forall v\in R,\\
		&&((q, v'), \epsilon, p, f, (q', v')) &&&\forall v'\in R' \,\mathrm{ and }\, (q, \epsilon, p, f, q')\in\delta,\\
		&&((q, a_1\dots a_n), \epsilon, p, f, (q', a_2\dots a_n)) &&&\forall a_1\dots a_n\in R' \,\mathrm{ and }\, (q, a_1, p, f, q')\in\delta,\\
		&&((q, \epsilon), \sigma, p, f, (q', \epsilon)) &&& \forall (q, \sigma, p, f, q')\in \delta.		
	\end{flalign*}
	The automaton will have start state $S$ and final state $(q_f, \epsilon)$.
\end{proof}

We stress once more that everything boils down to the fact that given an automaton \(\mathcal{M}\) and a finite number of words \(w_i \in \Sigma^\ast\), it is possible to insert a routine in the automaton that will mimic the behaviour of \(\mathcal{M}\) when a word \(w_i\) is read, that is, to "insert" \(w_i\) in the processed string of letters. The way it is done, is by adding the various suffixes of the \(w_i\) as a "second variable" to the states.

If $H$ is a normal subgroup $G$, then the word problem in $G/H$ is exactly the set of words representing elements of $H$. Thus we immediately get the following corollary.

\begin{corollary}
	If $G$ is a groups with multiple context-free word problem and $H$ is a finite normal subgroup of $G$, then $G/H$ has a multiple context-free word problem.
\end{corollary}

We recalled the following result from \cite{lyndon_combinatorial_2001}

\begin{theorem}[\cite{lyndon_combinatorial_2001} p.187, Theorem 2.6]\label{thm:normal_form_amalgams}
Let \(G = G_1 \ast_H G_2\) be an amalgamated product and let \(c_1, \dots, c_n\) be a sequence of elements of \(G\) such that:
\begin{enumerate}
\item \(n \geq 2\).
\item Each \(c_i\) is in one of the factors \(G_1\) or \(G_2\).
\item The words \(c_i\), \(c_{i+1}\) come form different factors.
\item No \(c_i\) is in \(H\).
\end{enumerate}
Then the product \(c_1 \dots c_n\) is non trivial in \(G\).
\end{theorem}

With Proposition \ref{prop:finitesubgroup} we can prove our main theorem, as previously stated the idea is similar to Theorem \ref{thm:free_product} with a few extra details. 

\begin{theorem}\label{thm: main result for amalgamation}
	Let $G_1$ and $G_2$ be groups whose word problem with multiple context-free. Let $H_i$ be a finite subgroup of $G_i$, such that $H_1\cong H_2\cong H$. Then $G = G_1\ast_{H}G_2$ has a multiple context-free word problem. 
\end{theorem}
\begin{proof}
	The idea is the following: suppose that the word \(w=a_1 \dots a_m\) is read. If all \(a_i\) are contained in only one of \(\Sigma_1\) or \(\Sigma_2\), 
	the automaton will then proceed as in Proposition \ref{prop:finitesubgroup} having guessed that it will read the trivial element. 
    So suppose this doesn't happen. 
	We can subdivide the word \(w\) into (maximal) subwords that contain only elements of \(\Sigma_1\) or \(\Sigma_2\). 
	This will give a sequence \(c_1, \dots , c_n\) of elements of \(G\). 
	Theorem \ref{thm:normal_form_amalgams} gives that \(w =_G c_1 \dots c_n\) represents the trivial element only if there is an
	\(i\) such that \(c_i\) represents an element of \(H\).  Let \(u\) be the subword of \(w\) associated to \(c_i\). Without loss of 
	generality, we may assume that \(u \in \Sigma_1^*\).
	By non-determinism, the automaton will guess the correct \(i\) and the element 	\(c_i \in H\). 
	Then, using the procedure detailed in Proposition \ref{prop:finitesubgroup}, it will  check if \(u\) really represents \(c_i\) and, if this is the case, the automaton will return to the point it started reading $u$ and proceed as if it had, instead, read the word \(v\in \Sigma^*_2\) representing \(c_i\) in \(G_2\).
	Note that for this last step it is crucial that \(H\) is finite. 
	
	It is clear that the word \(w\) will be accepted if and only if the automaton will accept the word obtained by \(w\) substituting 
	\(u\) with \(v\). 
	By induction on the length of the sequence \(c_1, \dots, c_n\), we get the result.

    	More formally:	
	Let $W_i$ be the word problem in $G_i$. Let $\mathcal{M}_i$ be an automaton accepting the language $W_i$. 
	Let $w_i^h$ be a word in $\Sigma_i$ representing the element $h\in H$. 
	Let $F_i = \{w_i^h\mid h\in H\}$ with a bijection $\phi\colon F_1\to F_2$ such that $\phi(w_1^h) = w_2^h$, let $\psi = \phi^{-1}$. Let $F_i'$ be the set of suffixes of words in $F_i$. Let $\mathcal{M}_i'$ be the automaton recognising words in $H_i$ from Proposition \ref{prop:finitesubgroup} with states $Q_i\times F_i'\sqcup \{S_i\}$. 
	
	Let $W$ be the word problem in $G$. We build an automaton similar to Theorem \ref{thm:free_product} accepting the language $W$.

	The states of $\mathcal{M}$ are $Q_1\times F_1'\sqcup Q_2\times F_2'\sqcup \{S_1, S_2, S, F\}$. The storage will be tree stacks over the alphabet $\Omega_1\sqcup\Omega_2\sqcup(Q_1\times F_1)\sqcup(Q_2\times F_2)\sqcup \{\square_1, \square_2\}$.
	
	The transitions will consist the following:
	
	\begin{enumerate}
		\item
			\(\left\{\left(S, \epsilon, C, \texttt{push}_1\left(\square_1\right), \left(q_I^1, \epsilon\right)\right), \left(S, \epsilon, C, \texttt{push}_1\left(\square_2\right), \left(q_I^2, \epsilon\right)\right)\right\}\cup\\	
			\left\{\left(S, \epsilon, \texttt{equals}(\root), \mathrm{Id}, F\right)\right\}\)
		\item  
			\(\left\{\left(\left(q_f^1, \epsilon\right), \epsilon, \texttt{equals}\left(\square_1\right), \texttt{down}, S\right), \left(\left(q_f^2, \epsilon\right), \epsilon, \texttt{equals}\left( \square_2\right), \texttt{down}, S\right)\right\}\)
		\item 
			\(	\big\{\left(\left(q, \epsilon\right), \epsilon, \mathtt{notequals}\left(Q_2 \times F_2\right), \mathtt{push}_i\left(\left(q, w\right)\right), \left(q_I^2, {\phi(w)}^{-1}\right) \right) \mid \\  q\in Q_1, w\in F_1, i\in \left\{-1, \dots, -n\right\}\big\}\) 
		\item 
			\(\big\{\left(\left(q, \epsilon\right), \epsilon, \mathtt{notequals}\left(Q_1 \times F_1\right), \mathtt{push}_i\left(\left(q, w\right)\right), \left(q_I^1, {\psi(w)}^{-1}\right)\right)\mid  \\  
			q\in Q_2, w\in F_2, i\in \left\{-1, \dots, -n\right\}\big\}\)
		\item 
			\(\left\{\left(\left(q_f^1, \epsilon\right), \epsilon, \mathtt{equals}\left(\left(q', w\right)\right), \mathtt{down}, \left(q', w\right)\right)\mid  q'\in Q_2, w\in F_2\right\}\cup \\
			\left\{\left(\left(q_f^2, \epsilon\right), \epsilon, \mathtt{equals}\left(\left(q', w\right)\right), \mathtt{down}, \left(q', w\right)\right)\mid  q'\in Q_1, w\in F_1\right\}\).
		\item 
			The transitions of $\mathcal{M}_i'$ except those with form \((S_i, \epsilon, \texttt{equals}(\root), \mathrm{Id}, (Q_I, v)).\)
	\end{enumerate}

	Before explaining in detail the rules, there is one key and central observation. If the automaton is in a state \((q, w)\) with \(w = a_1 \dots a_l \neq \epsilon\), then the only possible rules 	are those from group (6). In particular, by the definition of \(\mathcal{M}_i'\) (see the proof of Proposition \ref{prop:finitesubgroup}), the only such rules 	are of the form \(((q, a_1\dots a_l), \epsilon, p, f, (q', a_2\dots a_l))\), where \((q, a_1, p, f, q')\) was a rule of \(\mathcal{M}_i\) or \(((q, a_1\dots a_l), \epsilon, p, f, (q', a_1\dots a_l))\), where \((q, \epsilon, p, f, q')\) was a rule of \(\mathcal{M}_i\). That is, if there is a non empty word \(w\) at the second variable, the only possible rule that can be applied is one mimicking the behaviour of one of the original automata if the first letter of \(w\) was read. 
	That is, the priority is always to deplete the second variable of the states. 
	
	The elements of the group (1) consist of the very final instruction and the two instructions that starts processing letters in one of the two alphabets \(\Sigma_i\).
	
	The elements of the group (2) consist of the second to last move in a run, they are triggered when the complete word has been read and the tree-stack is one step away from the root.
	
	The elements of the groups (3) and (4) consists of the same type of rules, with the roles of \(G_1\) and \(G_2\) interchanged. 
	The rules describe the following instruction (say for the group (3)): "At any moment where the stack pointer is not pointing an element of \(Q_2 \times F_2\), and your state has empty second variable, you can guess that a sub-word that represents \(\phi(w)\) is starting, for some \(w \in H\). Then, you start a new branch and add \(\phi(w) ^{-1}\) at the second variable". If the guess was correct, then eventually the automaton will return to the root of the new branch with state \((q_2, \epsilon)\). Thus, it successfully 
	processed a sub-word that represented \(\phi(w)\). In this case, the rules of group (5) apply.
	Indeed, remember that, at the beginning of the process, we pushed \((q,w)\) in the stack, to remember the state at which the automaton was (as in Theorem \ref{thm:free_product}) and the word we were checking. Then, we put \(w\) in the second variable. What happened, is that we effectively substituted the sub-word representing \(\phi(w)\) with \(w\). 
	
	We will now give a precise proof of the Theorem. This automaton works similarly to the automaton in Theorem \ref{thm:free_product}. 
		
	Let $\Lambda$ be an accepted run for the automaton. Let $T_f$ be the final tree-stack for this run. We colour the non-root vertices of $T_f$ red and blue as in the proof of Theorem \ref{thm:free_product}. 
		
	There is a subtree $T_c\subset T_f$ of a single colour whose complement is connected. Assume that $T_c$ is a tree with labels from $\Omega_1$. For each instruction there are two possible pointers, these can be viewed as vertices of $T_f$. Let $\Theta$ be the subset of the instruction in $\Lambda$ such that both pointers are in $T_c$. It can be seen as in the proof of Theorem \ref{thm:free_product} that all these instructions are consecutive. Since $\Theta$ starts and ends at the root, the word read while performing the instructions in $\Theta$ represents an element $v\in F_1$. 
		
	The run $\Lambda$ decomposes as a concatenation $\Lambda_1\theta_1\Theta\theta_2\Xi\Lambda_2$, where $\theta_1$ and $\theta_2$ correspond to entering and leaving the tree $T_c$ and $\Xi$ is the run from $(q, \phi(v))$ to the first state $(q', \epsilon)$. 
			
	Since the tree $T_c$ cannot be reentered we see that $\Lambda$ is a valid run if and only if there is a valid run of the form $\Lambda_1\Theta'\Lambda_2$, where $\Theta'$ is the same run as $\Xi$ running through the states $(q, \epsilon)$ instead of $(q, w)$, one could see this as a run in $\mathcal{M}_2$ corresponding to $\Xi$. 
		
	The original decomposition $\Lambda_1\theta_1\Theta\theta_2\Xi\Lambda_2$ corresponds to a decomposition of $w$ as $u_1vu_2$. The word corresponding to the run $\Lambda_1\Theta'\Lambda_2$ is $u_1\phi(v)u_2$.
		
	It should be noted that the final tree for the run $\Lambda_1\Theta'\Lambda_2$ will have one fewer red subtree. 
		
	For the base case note that if $T_c = T_f$, then we have a word in $W_1\cup W_2$. Thus by induction on the number of maximal one-coloured subtrees, $L(\mathcal{M})$ is a subset of $W$. 	
	
	We must now prove that this automaton accepts all words in $W$. We will use the free product length of a word once again. Let $w = w_1\dots w_k$ be a word of free product length $k$. If this word represents the trivial word, then there is a subword $w_j$ which represents and element of $H$. Let $u$ be the corresponding element of $F_1$. We can assume this word is in $\Sigma_1^*$. Let $v$ be an element of $F_2$ representing the same element as $w_j$. 
	
	The automaton will leave the automaton $\mathcal{M}_2'$ from the state $(q, \epsilon)$ to the automaton $\mathcal{M}_1'$ starting at the state $(q_I^1, u)$. When the word $w_j$ is read the automaton will return to $\mathcal{M}_2'$ at the state $(q, v)$. The automaton will then make a run in $\mathcal{M}_2$ for the word $v$. Thus $w$ is in $L(\mathcal{M})$ if and only if $w' = w_1\dots w_{j-1}vw_{j+1}\dots w_k$ is in $L(\mathcal{M})$. Since $w'$ has shorter free product length and it is clear that words of free product length 1 are in $L(\mathcal{M})$, we are done by induction. 
\end{proof}

\section{HNN extensions and graphs of groups}
The goal of this section is to prove Theorem \ref{thm: main result for amalgamation} for HNN extension with finite associated subgroup. We recall the definition of HNN extension.
\begin{Def}[HNN extension]
	Let $G$ be a group, $H_1, H_2$ be two subgroups of \(G\)  and $\phi\colon H_1\to H_2$ be an isomorphism. The {\em HNN extension} is the group given by the presentation $G*_{\phi} = \langle G, t\mid tgt^{-1} = \phi(g)\mbox{ for all } g\in H_1\rangle$.
\end{Def}
Our goal is to prove the following result. 

\begin{theorem}\label{thm: main result for HNN extension}
 Let \(G\) be a finitely generated group whose word problem is multiple context-free. Let \(H_1\) and \(H_2\) be two finite subgroups of \(G\) and let \(\phi\colon H_1 \to H_2\) be an isomorphism. Then the HNN extension \(G \ast_{\phi}\) has a multiple context-free word problem.
\end{theorem}

The proof of Theorem \ref{thm: main result for HNN extension} almost conicides with the proof in the case of the amalgamated product, modulo the following lemma.

\begin{lemma}
	Consider a word $g_0t^{\epsilon_1}g_1t^{\epsilon_2}\dots g_n$ in an HNN extension where $g_i\in G$ and $\epsilon_i = \pm 1$. If $w =1 $, then
	\begin{itemize}
		\item either $n=0$ and $g_0 = 1$ in $G$;
		\item or $n>0$ and for some $i\in\{1, \dots, n-1\}$ one of the following holds:
		\begin{enumerate}
			\item $\epsilon_i = 1$ and $\epsilon_{i+1} = -1$ and $g_i\in H_1$;
			\item $\epsilon_i = -1$ and $\epsilon_{i+1} = 1$ and $g_i\in H_2$.
		\end{enumerate}
	\end{itemize} 
\end{lemma}
\begin{proof}[Proof of Theorem \ref{thm: main result for HNN extension}]
	The proof here is the similar to the proof of Theorem \ref{thm: main result for amalgamation}. Instead of changing automaton when we change alphabet we instead note that each time we read a $t$ or $t^{-1}$ the next word we read must be an element $g$ in $H_1$ or $H_2$ respectively. Since \(H_i\) are finite groups, we can recognise such words. After doing this we return to where we were and proceed with the instruction as if we had read $\phi(g)$ or $\phi^{-1}(g)$ respectively.
\end{proof}

We have now all the ingredients to prove Theorem \ref{thm:graphofgroups}:
\begin{thmm}
 Let $G$ be the fundamental group of a finite graph of groups. Assume that all the vertex groups have multiple context-free word problem and all the edges groups are finite. Then $G$ has multiple context-free word problem. 
\end{thmm}
\begin{proof}
Let \(\mathcal{T}\) be a spanning tree in the graph of the graph of groups. Applying inductively Theorem \ref{thm: main result for amalgamation}, we obtain that \(\pi_1(\mathcal{T})\) has a multiple context-free word problem. 
Since adding an edge between two vertices of a graph of groups corresponds to an HNN extension, by iteratively applying Theorem \ref{thm: main result for HNN extension} we obtain the result.
\end{proof}

\end{document}